\documentclass[11pt,a4paper]{article}

\usepackage[T1]{fontenc}
\usepackage[utf8]{inputenc}
\usepackage{lmodern}
\usepackage[a4paper,margin=25mm]{geometry}
\usepackage{amsmath,amssymb,amsfonts,amsthm}
\usepackage{enumitem}
\usepackage[authoryear]{natbib}
\setcitestyle{aysep={}}
\usepackage[colorlinks=true,linkcolor=blue,citecolor=blue,urlcolor=blue]{hyperref}
\hypersetup{
  pdftitle={Only Segmented Heavy Tails Can Produce a Light-Tailed Minimum},
  pdfauthor={Sergey Foss, Michael Scheutzow, Anton Tarasenko}
}

\newcommand{\Prob}{\mathbb{P}}
\newcommand{\E}{\mathbb{E}}
\newcommand{\R}{\mathbb{R}}
\newcommand{\ov}[1]{\overline{#1}}
\newcommand{\ee}{\mathrm{e}}

\theoremstyle{plain}
\newtheorem{thm}{Theorem}[section]
\newtheorem{cor}[thm]{Corollary}

\newtheorem{prp}[thm]{Proposition}

\theoremstyle{remark}

\newtheorem{rem}[thm]{Remark}
\newtheorem{transitionrem}[thm]{Remark}
\theoremstyle{definition}

\newtheorem{defn}[thm]{Definition}

\title{Only Segmented Heavy Tails Can Produce a Light-Tailed Minimum}
\author{Sergey Foss$^{1,2}$ \and Michael Scheutzow$^{3}$ \and Anton Tarasenko$^{2,*}$}
\date{}

\begin{document}
\maketitle
\begin{center}
\small
\begin{minipage}{\textwidth}
\centering
$^{1}$School of Mathematical and Computer Sciences (MACS), Heriot-Watt University,\\
Edinburgh Campus, Edinburgh, EH14 4AS, Scotland, UK\par\smallskip
$^{2}$Sobolev Institute of Mathematics, SB RAN,\\
Akademika Koptyuga Ave.\ 4, Novosibirsk, 630090, Russia\par\smallskip
$^{3}$Institut f\"ur Mathematik, Sekr.\ MA 7-5, Technische Universit\"at Berlin,\\
Stra\ss e des 17.\ Juni 136, Berlin, 10623, Berlin, Germany\par\smallskip
\href{mailto:s.foss@hw.ac.uk}{\texttt{s.foss@hw.ac.uk}}\quad
\href{mailto:ms@math.tu-berlin.de}{\texttt{ms@math.tu-berlin.de}}\quad
\href{mailto:tarasenko@math.nsc.ru}{\texttt{tarasenko@math.nsc.ru}}\par
$^{*}$Corresponding author
\end{minipage}
\end{center}

\begin{abstract}
A random variable $\xi$ has a {\it light-tailed} distribution (for
    short: is light-tailed) if it possesses a finite exponential moment, $\E
    \exp (\lambda \xi) < \infty$ for some $\lambda > 0$, and has a {\it
    heavy-tailed} distribution (is heavy-tailed) if $\E \exp (\lambda\xi) =
    \infty$, for all $\lambda > 0$.
    In \cite{LSK1}, the authors presented a particular example of a light-tailed
        random variable that is the minimum of two independent heavy-tailed
        random variables.
    In \cite{FKT}, it was shown that any light-tailed random variable
        with right-unbounded support may be represented as the minimum of two
        independent heavy-tailed random variables, with further generalisations
        of the result in a number of directions.

    We analyse an ``inverse'' question.
    Namely, we obtain necessary and sufficient conditions on the distribution
        of a heavy-tailed random variable, say $\xi_1$, that allow us to find
        another independent heavy-tailed random variable, say $\xi_2$, such that
        their minimum $\min (\xi_1, \xi_2)$ is light-tailed.
    We also provide a number of extensions of this result.
\end{abstract}



\section{Introduction and main results}

It was shown by example in \cite{LSK1} that the minimum of two independent
    heavy-tailed random variables may be light-tailed.
The example, however, is somewhat involved and does not fully explain
    the mechanism behind this phenomenon.

In \cite{FKT} a systematic construction for generating such distributions was
    proposed.
Moreover, a simple general scheme for constructing random variables
    with arbitrarily heavy tails whose minimum has an arbitrarily light tail was
    given.
The authors of \cite{FKT} also analyse a number of related problems.

In the current article, we look at the problem from a different angle.
Assume there is given a heavy-tailed random variable $\xi_1$.
What conditions on its distribution $F$ ensure that one can construct another
    independent heavy-tailed random variable, say $\xi_2$, such that
\[
    \eta \mathrel{ = } \xi_1 \wedge \xi_2
        \mathrel{ := } \min (\xi_1, \xi_2)
\]
is light-tailed?
More generally, what conditions on $F$ guarantee that the tail of $\eta$ is
    lighter than a prescribed distribution tail?
We establish necessary and sufficient conditions for both scenarios, and provide
    a number of further comments.

To formulate the results, let us recall some notation and notions.
We use the same notation $F$ for a probability distribution on the real line
    and for its distribution function $F(x)$, and we denote by
\[
    \ov{F}(x) = 1 - F(x)
\]
its right tail.

For a random variable $\xi$ with tail $\ov{F}_{\xi}(x) = \Prob(\xi >
    x)\mathrel{ > }0$, define its {\bf
    cumulative hazard} by
\[
    R_\xi(x) := -\log \ov{F}_\xi(x), \qquad x \in \R.
\]
If $\xi_1$ and $\xi_2$ are independent and $\eta = \xi_1 \wedge \xi_2$, then
\[
    \ov{F}_\eta(x) = \ov{F}_{\xi_1}(x)\, \ov{F}_{\xi_2}(x)\quad\text{and
        hence}\quad
    R_\eta(x) = R_{\xi_1}(x) + R_{\xi_2}(x), \ x \in \R.
\]

In general, we will say that a function $R:\R \to [0, \infty)$ is a {\bf
    cumulative hazard (function)} if $R$ is nondecreasing, right-continuous
    and satisfies
\[
    \lim_{x \to - \infty} R(x) = 0
    \quad\text{and}\quad
    \lim_{x \to \infty} R(x) = \infty,
\]
which is equivalent to saying that $R = R_\xi$ for some random variable $\xi$
    with right-unbounded support.

\begin{defn}
    A distribution $F$ (or a random variable $\xi$ following distribution $F$)
        has a {\bf heavy tail} (is {\bf heavy-tailed}) if, for all $\lambda >
        0$,
    \[
        \int_{-\infty}^{\infty} e^{\lambda x} F(dx)
        \mathrel{ = } \mathbb{E} e^{\lambda \xi} = \infty,
    \]
    and has a {\bf light tail} (is {\bf light-tailed}) if $\E e^{\lambda \xi} <
        \infty$, for some $\lambda > 0$.
\end{defn}

\begin{rem}\label{Rem1}
    Recall that a random variable $X$ is heavy-tailed if
    \[
        \liminf_{x \to \infty}\frac{R_X(x)}{x} = 0,
    \]
    and light-tailed otherwise -- see, e.g., \cite[Theorem~2.6,
        p.~8]{FKZ}.
\end{rem}

{ Heavy-tailedness requires the cumulative hazard to become
    arbitrarily small relative to $x$ at some points, while allowing it to have
    infinitely many segments where it is comparable to $x$ or even larger.
    For two independent variables, the cumulative hazard of their minimum is
        the sum of their hazards.
    These segments can be arranged so that whenever one hazard is small relative
        to $x$, the other compensates for it, keeping their sum comparable
        to $x$.
    Both variables are then heavy-tailed, but their minimum is light-tailed.
    The segmentation condition below characterises precisely when a prescribed
        heavy-tailed variable admits such a complement.
}

We will say that a sequence $1 < s_1 < t_1 < s_2 < t_2 < \ldots$ increasing
    to infinity forms a {\bf segmentation} (of the positive half-line) if
\begin{align}\label{eq:C2}
    \lim_{i \to \infty}\frac{s_i}{t_i} = 0.
\end{align}

\begin{defn}
    A distribution $F$ (or a random variable $\xi$ following distribution $F$)
        has a {\bf segmented heavy tail} (belongs to the class {\bf SHT}) if it
        is heavy-tailed and satisfies the following condition:

    \medskip
    \noindent\textbf{(C)} There exist $\gamma > 0$ and a segmentation $1 < s_1 <
        t_1 < s_2 < t_2 < \ldots$ such that
    \[
        \frac{R_{\xi}(x)}{x} \geqslant \gamma
        \quad\text{for all }x \in [s_i, t_i], \qquad i = 1, 2, \dots .
    \]
\end{defn}

{ The condition $s_i / t_i \to 0$ means that the intervals $[s_i,
    t_i]$ extend over increasingly large multiplicative ranges.
    Condition \textbf{(C)} requires the hazard to stay above a fixed linear
        bound throughout these intervals.
    Corollary~\ref{cor:limsup_infty} shows that even occasional arbitrarily
        large values of $R_{\xi_1}(x) / x$ are sufficient.
}

Here is our first main result.

\begin{thm}\label{thm:main_C}
    Let $\xi_1$ be a heavy-tailed random variable.
    The following are equivalent:
    \begin{itemize}
        \item[(i)] There exists a heavy-tailed random variable $\xi_2$,
            independent of $\xi_1$, such that $\eta := \xi_1 \wedge \xi_2$ is
            light-tailed;
        \item[(ii)] $\xi_1$ has a segmented heavy tail.
    \end{itemize}
\end{thm}

\begin{cor}\label{cor:limsup_infty}
    If
    \begin{align}\label{suf1}
        \limsup_{x \to \infty}\frac{R_{\xi_1}(x)}{x} = \infty,
    \end{align}
    then $\xi_1$ satisfies condition \textbf{(C)}.
\end{cor}

{
\begin{transitionrem}\label{rem:characterisation}
Theorem~\ref{thm:main_C} gives a complete answer to the question when a given
    heavy-tailed $\xi_1$ admits an independent heavy-tailed $\xi_2$ such that
    $\xi_1 \wedge \xi_2$ becomes light-tailed.
\end{transitionrem}
}
\noindent Now we extend this criterion to comparisons with a prescribed tail scale.

\begin{defn}
    Let $F$ be a distribution with right-unbounded support and let $R$ be its
        associated cumulative hazard function.
    A right-unbounded random variable $\xi$ with cumulative hazard $R_\xi$ is
        called {\bf $R$-heavy} if
    \[
        \liminf_{x \to \infty}\frac{R_\xi(x)}{R(x)} = 0,
    \]
    and {\bf $R$-light}, otherwise.
\end{defn}

\begin{rem}\label{Rem2}
    The cumulative hazard $R(x) = x \vee 0$ corresponds to
        the standard exponential distribution and recovers the standard notion
        of light/heavy tails, see Remark \ref{Rem1}.
\end{rem}

Let $R_\downarrow(x)$ and $R_\uparrow(x)$ be two cumulative hazard functions.
We give necessary and sufficient conditions on the law of
    an $R_\downarrow$-heavy random variable $\xi_1$ under which there exists
    an independent $R_\downarrow$-heavy random variable $\xi_2$ such that
    $\eta := \xi_1 \wedge \xi_2$ is $R_\uparrow$-light.

\begin{defn}
    \label{cond:segmentation} Let $R_\downarrow$ and $R_\uparrow$ be cumulative
        hazard functions such that $R_\downarrow(x) \leqslant R_\uparrow(x)$
        for all $x \geqslant x_0$, for some $x_0 \in \mathbb{R}$.
    A sequence of real numbers $s_1 < t_1 < s_2 < t_2 < \cdots$,
        tending to infinity, is called
        an {\bf $(R_\downarrow, R_\uparrow)$-segmentation} if $R_\downarrow(s_1)
        > 1$ and
    \begin{equation}
        \label{eq:segmentation_condition}
        \lim_{i \to \infty}\frac{R_\uparrow(s_i - 0)}{R_\downarrow(t_i)} = 0.
    \end{equation}
\end{defn}

\begin{thm}
    \label{proposition:main} Let $R_\downarrow$ and $R_\uparrow$ be cumulative
        hazard functions with $R_\downarrow(x) \leqslant R_\uparrow(x)$ for all
        $x \geqslant x_0$, for some $x_0 \in \mathbb{R}$, and let $\xi_1$ be
        an $R_\downarrow$-heavy random variable.
    Then the following are equivalent:
    \begin{itemize}
        \item[(i)] There exists, on a suitable probability space,
            an $R_\downarrow$-heavy random variable $\xi_2$,
            independent of $\xi_1$, such that $\eta := \xi_1 \wedge \xi_2$ is
            $R_\uparrow$-light.
        \item[(ii)] There exist $\gamma > 0$ and an $(R_\downarrow,
            R_\uparrow)$-segmentation $s_1 < t_1 < s_2 < t_2 < \cdots$ such that
        \[
            \frac{R_{\xi_1}(x)}{R_\uparrow(x)} \geqslant \gamma
            \qquad\text{for all }x \in [s_i, t_i], \ i = 1, 2, \dots .
        \]
    \end{itemize}
\end{thm}

\begin{cor}\label{Cor2}
    Under the assumptions of Theorem \ref{proposition:main}, if in addition
        $R_\downarrow \mathrel{ = } R_\uparrow$,
        $R_\uparrow$ is continuous, and
    \begin{equation}
        \label{cond:supremum_condition}
        \limsup_{x \to \infty}\frac{R_{\xi_1}(x)}{R_\uparrow(x)} = \infty,
    \end{equation}
    {then, for every $\gamma > 0$, there exists an
        $(R_\uparrow, R_\uparrow)$-segmentation
        $s_1 < t_1 < s_2 < t_2 < \cdots$ such that
    \[
        R_{\xi_1}(x) \geqslant \gamma R_\uparrow(x)
        \qquad\text{for all }x \in [s_i, t_i],\quad i = 1, 2, \dots .
    \]
    }
\end{cor}

We can also extend this result by considering not one, but $n - 1$ complementary
    random variables such that only the minimum of $k \leqslant n$ or more
    of them has a light tail.

{ For a finite nonempty index set $I$, we use the notation
    \[
        \bigwedge_{i \in I}\xi_i := \min_{i \in I}\xi_i.
    \]
}

\begin{thm}[$k$-out-of-$n$ light minimum]
    \label{thm:k_out_of_n_segmentation} Let $R_\downarrow$ and $R_\uparrow$ be
        cumulative hazards such that
    \[
        R_\downarrow(x) \leqslant R_\uparrow(x), \qquad x \geqslant x_0,
    \]
    and let $\xi_1$ be an $R_\downarrow$-heavy random variable.

    Fix $n \geqslant 2$ and $1 < k \leqslant n$.
    The following are equivalent:
    \begin{enumerate}
        \item[(i)] There exist $R_\downarrow$-heavy random variables $\xi_2,
            \dots, \xi_n$, independent of $\xi_1$ and each other, such that
            for every $I\subseteq\{1, \dots, n\}$,
        \[
            \bigwedge_{i \in I}\xi_i \text{ is $R_\uparrow$-light if } |I| = k,
            \qquad
            \bigwedge_{i \in I}\xi_i \text{ is $R_\downarrow$-heavy if } |I| = k
                - 1.
        \]

        \item[(ii)] There exist $\gamma > 0$ and an $(R_\downarrow,
            R_\uparrow)$-segmentation $s_1 < t_1 < s_2 < t_2 < \cdots$ such that
        \[
            \frac{R_{\xi_1}(x)}{R_\uparrow(x)} \geqslant \gamma
            \qquad\text{for all }x \in [s_i, t_i], \ i = 1, 2, \dots .
        \]
    \end{enumerate}
\end{thm}

\begin{rem}\label{Rem3}
    A key point is that the condition
        in Theorem~\ref{thm:k_out_of_n_segmentation} does not depend on
        the specific values of $k$ and $n$, as long as $1 < k \leqslant n$.
    That is, the same condition is necessary and sufficient for all admissible
        pairs $(k, n)$.
    This follows from the fact that once the statement is established in
        the case $n = 2$, the general case can be obtained by a decomposition
        argument: a single random variable $\xi_2$ may be replaced by several
        independent random variables constructed via a suitable resegmentation,
        while preserving the required tail properties.
\end{rem}

\begin{rem}\label{rem:tail_properties}
    Recall that we deal with distributions having right-unbounded support, i.e.,
        with distribution functions $F$ such that ${\overline F}(x) > 0$ for all
        $x$.

    We say that two distribution functions, $F$ and $G$, are {\bf
        tail-equivalent} if
    \begin{align}\label{taileq}
        \frac{\overline{F}(x)}{\overline{G}(x)} \to 1, \quad \text{as}
        \quad x \to \infty.
    \end{align}
    We say that a certain property is a {\bf tail property} if, when it holds
        for a distribution $F$ (or a random variable with distribution $F$), it
        also holds for any distribution $G$ that is tail-equivalent to $F$ (or
        for the corresponding random variable).

    Clearly, for any random variable $\xi$ with right-unbounded support and any
        number $x_0$, the distributions of $\xi$ and $\max (\xi, x_0)$ are
        tail-equivalent.
    Furthermore, heavy-tailedness is a tail property, as is SHT-ness and all
        other properties we discuss in this paper.

    Therefore, we may deal with non-negative random variables only.
    For example, a random variable $\xi_2$ may be chosen to be non-negative (or,
        more generally, greater almost surely than any predefined constant)
        in
        Theorem~{\ref{thm:main_C}}.
\end{rem}

\section{Discussion}

This section contains several complementary remarks.
In particular, we present a truncation counterexample showing that
    the sufficient conditions \eqref{suf1} and \eqref{cond:supremum_condition}
    are not necessary, provide an explicit construction illustrating condition
    \textbf{(C)}, and include further comments on how the class SHT relates
    to standard subclasses of heavy-tailed distributions.
We also discuss the assumptions on the hazard scales, the use
    of left limits, and comparisons on other tail scales.

\subsection{SHT and some common classes of heavy-tailed
    distributions} \label{ss23}

{  We will first explain
    the place of SHT among usual classes of heavy-tailed distributions.
    Heavy-tailedness alone doesn't say anything about variation of
        the distribution's tail.
    However most standard subclasses impose additional regularity conditions.
    For example, long-tailedness means that a fixed shift has an asymptotically
        negligible relative effect on the tail, while dominated variation limits
        the relative decrease in the tail caused by doubling the argument.
    The precise definitions are recalled below.
    Write $\mathcal{RV}$ for regularly varying distributions, $\mathcal D$
        for dominated-varying distributions, $\mathcal S$ for subexponential
        distributions, $\mathcal L$ for long-tailed distributions, and $\mathcal
        H$ for heavy-tailed distributions.
    For nonnegative distributions, the following relations between these
        standard classes are well-known;
        see, for instance, \citet{EKM} and \citet{FKZ}:
    \[
        \mathcal{RV}\subset\mathcal L\cap\mathcal D
        \subset\mathcal S\subset\mathcal L\subset\mathcal H.
    \]
    Separately, dominated variation implies heavy-tailedness:
    \[
        \mathcal D\subset\mathcal H.
    \]
    The propositions below establish that
    \[
        \mathcal{SHT}\cap(\mathcal L\cup\mathcal D) = \varnothing.
    \]
    In particular, SHT distributions do not belong to the classical regularly
        varying or subexponential classes in the above chains.
}

{\bf Long-tailed distributions are not segmented heavy-tailed.}
Long-tailed distributions play a central role in heavy-tail
    theory.
    They include regularly varying, lognormal, and heavy-tailed Weibull
        distributions and, more generally, all subexponential distributions.
Let us recall the corresponding definition.

\begin{defn}
    A distribution $F$ is {\bf long-tailed} (belongs to the class ${\cal L}$) if
        it has a right-unbounded support (i.e. $F(x) < 1$ for all $x$) and
    \begin{align}\label{long1}
        \lim_{x \to \infty} \frac{\ov{F}(x + 1)}{\ov{F}(x)} = 1,
    \end{align}
    or, equivalently,
    \begin{align}\label{long2}
        \lim_{x \to \infty} (R(x + 1) - R(x)) = 0,
    \end{align}
    where $R(x) =-\log \ov F(x)$ is the corresponding cumulative hazard.
\end{defn}

It is known (see e.g.
    {\citealp[Lemma~1.3.5(b) and Remark~4,
    pp.~41--42]{EKM} and }\citealp[Lemma~2.17, p.~19]{FKZ})
    that \emph{any long-tailed distribution is heavy-tailed}.

\begin{prp}\label{rem:not_long_tailed}
    If $\xi_1$ is segmented heavy-tailed, then its distribution cannot be
        long-tailed.
\end{prp}
Indeed, assume that a distribution $F$ with cumulative hazard $R$ is
    long-tailed.
Then it follows from \eqref{long2} that
\[
    \frac{R(n)}{n} = \frac{\sum_{i = 1}^n (R(i) - R(i - 1))}{n} +
        \frac{R(0)}n \to 0
\]
as $n \to \infty$ and then, for real $x$, we have $R(x) / x \to 0$ too,
    by monotonicity of $R(x)$.
Then condition \textbf{(C)} fails.

Note that the minimum $\eta = \xi_1 \wedge \xi_2$ of two independent long-tailed
    random variables must be heavy-tailed (because it is long-tailed, too).
It follows from Proposition \ref{rem:not_long_tailed} that this fact is more
    general:

\begin{cor}\label{Cor3}
    The minimum $\xi_1 \wedge \xi_2$ of a long-tailed random variable $\xi_1$
        and any independent heavy-tailed random variable $\xi_2$ is necessarily
        heavy-tailed.
\end{cor}

\begin{rem}
    A more general class ${\cal OL}$ of {\bf generalised long-tailed}
        distributions,
        introduced by \citet{SW}, is defined as follows:
    \begin{align}
        F \in {\cal OL} \ \ \text{if} \ \ \limsup_{x \to \infty}
        \frac{\overline{F}(x - 1)}{\overline{F}(x)} < \infty.
    \end{align}
    We show in
        Subsection~\ref{ss22} that, in contrast to long-tailed distributions,
        ${\cal OL}$ distributions may be ${\cal SHT}$.
\end{rem}

We now establish an analogous incompatibility for another standard heavy-tailed
    subclass, dominated-varying distributions.

\vspace{0.3cm}
{\bf Dominated-varying distributions are not segmented heavy-tailed.}\par
Another classical subclass of heavy-tailed distributions is formed by
    the dominated-varying distributions.

\begin{defn}
    A distribution $F$ is {\bf dominated-varying} (belongs to the class ${\cal
        D}$) if $F(x) < 1$ for all $x$ and there exists $c > 0$ such that
    \begin{equation}\label{eq:DV_def_tail}
        \ov F(2x) \geqslant c\, \ov F(x)
        \qquad\text{for all }x \geqslant 1.
    \end{equation}
    In terms of the cumulative hazard $R(x) =-\log \ov F(x)$, condition
        \eqref{eq:DV_def_tail} is equivalent to the existence of $p > 0$ such
        that
    \begin{equation}\label{eq:DV_def_haz}
        R(2x) \leqslant R(x) + p
        \qquad\text{for all }x \geqslant 1.
    \end{equation}
\end{defn}
It is known (see e.g.
    {\citealp[Section~1.4.1, pp.~49--50,
    especially Fig.~1.4.1]{EKM} and }\citealp[Section~2.1,
    p.~11]{FKZ}) that \emph{any dominated-varying distribution is
    heavy-tailed}.

We show that dominated variation is incompatible with condition \textbf{(C)}.

\begin{prp}\label{prp:not_dominated_varying}
    If $\xi_1$ is segmented heavy-tailed, then its distribution cannot be
        dominated-varying.
\end{prp}

Indeed, assume that \eqref{eq:DV_def_haz} holds for some $p > 0$.
Iterating this bound, we obtain for every $n \in \mathbb{N}$,
\begin{equation}\label{eq:DV_iter}
    R(2^{n + 1})
    \leqslant R(2^n) + p
    \leqslant \cdots
    \leqslant R(1) + (n + 1)p.
\end{equation}
Now fix $n \in \mathbb{N}$ and take any $x \in [2^n, 2^{n + 1}]$.
By monotonicity of $R$,
\[
    0 \leqslant \frac{R(x)}{x}
    \leqslant \frac{R(2^{n + 1})}{2^n}
    \leqslant 2^{-n}\bigl(R(1) + (n + 1)p\bigr).
\]
The right-hand side converges to $0$ as $n \to \infty$, and therefore $R(x) /
    x \to 0$ as $x \to \infty$.
Hence no $\gamma > 0$ can satisfy
\[
    \frac{R(x)}{x} \geqslant \gamma
    \quad\text{for all }x \in [s_i, t_i]\text{ along any segmentation }(s_i,
        t_i),
\]
so condition \textbf{(C)} cannot hold.

\subsection{Why an eventual ordering between the two hazard scales is needed}
The requirement that $R_\downarrow(x) \leqslant R_\uparrow(x)$ holds for all
    sufficiently large $x$ (or some other comparable asymptotic restriction) is
    not merely cosmetic.
If $R_\downarrow$ eventually dominates $R_\uparrow$, then
    $R_\downarrow$-heaviness becomes too weak a constraint: it may already
    include random variables that are $R_\uparrow$-light.
In that case, condition (i) can hold for a trivial reason, even when $\xi_1$ has
    no segmented behaviour on the $R_\uparrow$-scale.
Therefore, the equivalence in Theorem~\ref{proposition:main} cannot be true
    without an additional relation between the scales.

A concrete example is obtained by taking
\[
    R_\downarrow(x) = (x \vee 0)^2,
    \qquad
    R_\uparrow(x) = x \vee 0.
\]
Let $\xi_1$ have a regularly varying tail, say $\Prob(\xi_1 > x) = x^{-\alpha}$
    for $x \geqslant 1$, for some given $\alpha > 0$, and let $\xi_2$ be
    exponentially distributed with $\Prob(\xi_2 > x) = e^{-x}$ for $x \geqslant
    0$.
Then, for $x \geqslant 1$,
\[
    R_{\xi_1}(x) = \alpha\log x,
    \qquad
    R_{\xi_2}(x) = x,
\]
and hence both $\xi_1$ and $\xi_2$ are $R_\downarrow$-heavy since
\[
    \frac{R_{\xi_1}(x)}{R_\downarrow(x)} = \frac{\alpha\log x}{x^2} \to 0,
    \qquad
    \frac{R_{\xi_2}(x)}{R_\downarrow(x)} = \frac{x}{x^2} \to 0.
\]
For $\eta = \xi_1 \wedge \xi_2$ we have $R_\eta = R_{\xi_1} + R_{\xi_2}$, so
\[
    \frac{R_\eta(x)}{R_\uparrow(x)}
    =
    \frac{\alpha\log x + x}{x}
    \to 1,
\]
and therefore $\eta$ is $R_\uparrow$-light.
Thus condition (i) holds.

On the other hand, $\xi_1$ cannot satisfy condition (ii), because
\[
    \frac{R_{\xi_1}(x)}{R_\uparrow(x)} = \frac{\alpha\log x}{x} \to 0,
\]
so the lower bound $R_{\xi_1}(x) / R_\uparrow(x) \geqslant \gamma$ cannot hold
    on any sequence of intervals $[s_i, t_i]$ with $t_i \to \infty$.
This illustrates why some eventual constraint linking $R_\downarrow$
    and $R_\uparrow$ is needed.

\subsection{The left limit is essential in the definition
    of segmentation}\label{leftlimit}

We give an example showing that, in the definition of $(R_\downarrow,
    R_\uparrow)$-segmentation, the left limit $R_\uparrow(s_i - 0)$
    in \eqref{eq:segmentation_condition} cannot, in general, be replaced by
    the right limit $R_\uparrow(s_i)$.
That is, if one defines segmentation via
\[
    \lim_{i \to \infty}\frac{R_\uparrow(s_i)}{R_\downarrow(t_i)} = 0
\]
instead of \eqref{eq:segmentation_condition}, then the implication
    ``(i)$\Rightarrow$(ii)'' in Theorem~\ref{proposition:main} may fail.

\smallskip
Fix a strictly increasing sequence $(A_n)_{n \geqslant 1}$ such that $A_n \to
    \infty$ and $A_n \geqslant 2$, and define $(B_n)_{n \geqslant 0}$ by $B_0 :=
    1$ and $B_n := A_n^2B_{n - 1}$.
Let $\mathbb{Z}_+$ be the set of nonnegative integers and $\mathbb{N}$ the set
    of natural numbers.
Introduce two piecewise constant cumulative hazards by
\[
    R_\uparrow(x) :=
    \begin{cases}
        0, & x < 0, \\
        B_n, & x \in [n, n + 1), \ n \in \mathbb{Z}_{+},
    \end{cases}\ 
    R_\downarrow(x) :=
    \begin{cases}
        0, & x < 0, \\
        1, & x \in [0, 1), \\
        A_nB_{n - 1}, & x \in [n, n + 1), \ n \in \mathbb{N}.
    \end{cases}
\]
Both functions are nondecreasing, right-continuous, and diverge to infinity, so
    they are cumulative hazards.
Moreover, for $x \in [n, n + 1)$ with $n \geqslant 1$,
\[
    \frac{R_\downarrow(x)}{R_\uparrow(x)}
    = \frac{A_nB_{n - 1}}{B_n}
    = \frac1{A_n},
\]
and hence $R_\downarrow(x) \leqslant R_\uparrow(x)$ for all $x \geqslant 1$.

\smallskip
Next, define two further cumulative hazards by alternating between ``growing
    like $R_\uparrow$'' and ``freezing'':
\[
    R_1(x) :=
    \begin{cases}
        0, & x < 0, \\
        B_{2k}, & x \in [2k, 2k + 2), \ k \in \mathbb{Z}_{+},
    \end{cases}\ 
    R_2(x) :=
    \begin{cases}
        0, & x < 0, \\
        1, & x \in [0, 1), \\
        B_{2k - 1}, & x \in [2k - 1, 2k + 1), \ k \in \mathbb{N}.
    \end{cases}
\]
Thus $R_1$ coincides with $R_\uparrow$ on even unit intervals and is frozen
    on odd ones, while $R_2$ does the opposite.
Let $\xi_1, \xi_2$ be independent random variables with $R_{\xi_1} = R_1$
    and $R_{\xi_2} = R_2$.

\smallskip
We first note that both $\xi_1$ and $\xi_2$ are $R_\downarrow$-heavy.
Indeed, on odd unit intervals $[2k + 1, 2k + 2)$ we have $R_{\xi_1}(x) = B_{2k}$
    and $R_\downarrow(x) = A_{2k + 1}B_{2k}$, so
\[
    \frac{R_{\xi_1}(x)}{R_\downarrow(x)} = \frac1{A_{2k + 1}}\xrightarrow[k \to
        \infty]{}0.
\]
Similarly, on even unit intervals $[2k, 2k + 1)$ one has $R_{\xi_2}(x) /
    R_\downarrow(x) = 1 / A_{2k} \to 0$.

Now consider $\eta := \xi_1 \wedge \xi_2$.
Since $R_\eta = R_{\xi_1} + R_{\xi_2}$, and on each unit interval $[n, n + 1)$
    exactly one of $R_{\xi_1}, R_{\xi_2}$ equals $R_\uparrow$ while the other is
    nonnegative, it follows that $R_\eta(x) \geqslant R_\uparrow(x)$ for all
    $x \geqslant 1$.
Consequently,
\[
    \liminf_{x \to \infty}\frac{R_\eta(x)}{R_\uparrow(x)} \geqslant 1,
\]
so $\eta$ is $R_\uparrow$-light.
Thus condition~(i) holds.

\smallskip
Assume that $s_1 < t_1 < s_2 < \cdots$ is a sequence which diverges to infinity
    and for which there exists some $\gamma > 0$ such that
$\frac {R_{\xi_1}(x)}{R_\uparrow(x)} \geqslant \gamma$ for all $x \in [s_i,
    t_i]$, $i = 1, 2, \cdots$.
We have
\[
    \frac{R_{\xi_1}(2k + 1)}{R_\uparrow(2k + 1)}
    = \frac{B_{2k}}{B_{2k + 1}}
    = \frac1{A_{2k + 1}^2}\xrightarrow[k \to \infty]{}0.
\]
Using the fact that all cumulative hazards involved are piecewise constant we
    see that there exists some $k_0 \geqslant 1$ such that $$ [s_i,t_i] \cap
    \bigcup_{k=k_0}^\infty[2k-1,2k)=\emptyset \mbox{ for all } i\geqslant 1. $$
    Hence, for sufficiently large $i$, there exists some $k = k(i)$ such that
    $[s_i, t_i]\subseteq [2k, 2k + 1)$ and so
\[
    \frac{R_\uparrow(s_i)}{R_\downarrow(t_i)}
    = \frac{B_{2k}}{A_{2k}B_{2k - 1}}
    = A_{2k}\xrightarrow[k \to \infty]{} + \infty,
\]
so the modified condition (ii), with $R_\uparrow(s_i)$ in place of
    $R_\uparrow(s_i - 0)$, fails.

\subsection{Non-necessity of sufficient conditions in Corollaries
    \ref*{cor:limsup_infty}
    and \ref*{Cor2}}\label{ss21}

\begin{prp}
    For $R_\downarrow = R_\uparrow$,
        condition (ii) of Theorem~\ref{proposition:main} may hold even when
        \eqref{cond:supremum_condition} fails.
\end{prp}
Indeed, let $R_\uparrow$ be continuous and consider an $R_\uparrow$-heavy random variable $\xi_1$ such that
\[
    \limsup_{x \to \infty}\frac{R_{\xi_1}(x)}{R_\uparrow(x)} = \infty.
\]
Then, by Corollary~\ref{Cor2}, for some $\gamma > 0$ there
    exists an $(R_\uparrow, R_\uparrow)$-segmentation on which
    $R_{\xi_1}(x) \geqslant \gamma R_\uparrow(x)$.
Define a modified cumulative hazard by
\[
    R(x) := R_{\xi_1}(x) \wedge \frac{\gamma}{2}\, R_\uparrow(x), \qquad x \in
        \mathbb{R}.
\]
The function $R$ is nondecreasing, right-continuous, and diverges to infinity,
    hence it is the cumulative hazard of some random variable $\zeta$.
Moreover, $\zeta$ is $R_\uparrow$-heavy since $R(x) \leqslant R_{\xi_1}(x)$
    and $\xi_1$ was $R_\uparrow$-heavy.
Finally, on every interval where $R_{\xi_1}(x) / R_\uparrow(x) \geqslant
    \gamma$, we also have $R(x) / R_\uparrow(x) \geqslant \gamma / 2$.
Thus the same segments give the required bound for $\zeta$,
    with $\gamma / 2$ in place of $\gamma$, while
\[
    \limsup_{x \to \infty}\frac{R(x)}{R_\uparrow(x)} \leqslant \frac{\gamma}{2}
        < \infty.
\]
Thus \eqref{cond:supremum_condition} is a sufficient but not necessary
    condition.

\begin{rem}
    On the other hand, it is clear that if for {\bf every}
        $\gamma > 0$ there exists an $(R_\uparrow, R_\uparrow)$-segmentation
        on which $R_{\xi_1}(x) \geqslant \gamma R_\uparrow(x)$, then
        \eqref{cond:supremum_condition} holds.
\end{rem}

\begin{cor}
    Condition \textbf{(C)} is condition
        (ii) of Theorem~\ref{proposition:main} in the particular case
        $R_\downarrow(x) = R_\uparrow(x) = x \vee 0$.
    Therefore, the truncation argument above immediately implies that
        the sufficient condition in Corollary~\ref{cor:limsup_infty} is not
        necessary for \textbf{(C)}.
\end{cor}

\subsection{Construction of an SHT distribution based on a light-tailed
    distribution}\label{ss22}
We present a construction of an SHT distribution on the positive half-line
    by segmenting a light-tailed exponential distribution with the tail
    $e^{-x}$, for $x \geqslant 0$.
A slightly modified construction can be given in the case of the tail
    $e^{-x^{\beta}}$ with $\beta > 1$.

Take any $\gamma \in (0, 1)$, any strictly decreasing sequence $\{u_n\}$ tending
    to zero with $u_1 < \gamma$, and any strictly increasing sequence $\{v_n\}$
    diverging to infinity with $1 < v_1$.
We recursively construct a strictly increasing to infinity sequence
\begin{align*}
    t_0 = 1 < a_1 < s_1 < t_1 < a_2 < s_2 < t_2 < \ldots < a_k < s_k < t_k <
        \ldots
\end{align*}
and a function $R(x)$ defined as $R(x) = 0$ for $x < 0$, $R(x) = x$ for
    $0 \leqslant x \leqslant 1$ and, for any $k \geqslant 1$,
\begin{align*}
    R(x) &= R(t_{k - 1}) = R(a_k) \quad \text{for} \quad t_{k - 1} \leqslant
        x \leqslant a_k;
    \\
    R(x) &= R(a_k) + (x - a_k) \quad \text{for} \quad a_k \leqslant x \leqslant
        t_k.
\end{align*}

Here is the recursive scheme.
We let
\begin{align*}
    a_k = \frac{R(t_{k - 1})}{u_k}, \quad
    s_k = \frac{a_k - R(a_k)}{1 - \gamma} \mathrel{ = }
        \frac{a_k(1 - u_k)}{1 - \gamma}, \quad
    t_k = v_k s_k.
\end{align*}
Indeed, the distribution is heavy-tailed since $R(a_k) / a_k \to 0$.
Further, both parts of condition \textbf{(C)} hold.

{
It also follows that $F\in\mathcal{OL}$.
Indeed, the cumulative hazard $R$ is continuous and piecewise linear with
    slopes $0$ and $1$, so
\[
    \frac{\overline F(x-1)}{\overline F(x)}
    = e^{R(x)-R(x-1)}\leqslant e,\qquad x\geqslant1.
\]
}

\subsection{Further comments}
Condition \textbf{(C)} is a special case of Theorem~\ref{proposition:main}
    with $R_\downarrow(x)\mathrel{ = } R_\uparrow(x) = x \vee
    0$, where the target corresponds to classical light-tailedness.
The two-scale formulation makes it possible to compare tails relative to
    a chosen scale, which is useful when the dichotomy ``heavy vs.\ light'' is
    too coarse.
Below we present several illustrative examples.

(1) \emph{Two heavy-tail classes.}
Let $R_\downarrow(x) = x^{\beta_1}$ and $R_\uparrow(x) = x^{\beta_2}$ with $0 <
    \beta_1 < \beta_2 < 1$.
Both correspond to classically heavy tails.
However, $(R_\downarrow, R_\uparrow)$-segmentation characterises when
    an $R_\downarrow$-heavy random variable $\xi_1$ admits an independent
    $R_\downarrow$-heavy $\xi_2$ random variable such that $\xi_1 \wedge \xi_2$
    is $R_\uparrow$-light, that is, has a tail at least on the scale
    $\exp(-x^{\beta_2})$.
Condition \textbf{(C)} is not designed for such within-heavy comparisons.

(2) \emph{Two light-tail classes.}
To distinguish ordinary light tails from Gaussian-type tails, one may take
    $R_\uparrow(x) = x^2$.
Then distributions with $R_\xi(x) = o(x^2)$ (e.g.\ exponential tails) are
    $R_\uparrow$-heavy, whereas $R_\uparrow$-lightness corresponds to tails
    eventually bounded above by $\exp(-c x^2)$ for some $c > 0$.
Theorem~\ref{proposition:main} provides a criterion for when a complement
    $\xi_2$ can make $\xi_1 \wedge \xi_2$ $R_\uparrow$-light.
Condition \textbf{(C)} does not apply here, since $\xi_1$ need not be
    heavy-tailed in the classical sense.

(3) \emph{A pronounced gap between input and target scales.}
One can impose a substantial separation of scales, for example
    $R_\downarrow(x) = \log(1 + x)$ and $R_\uparrow(x) = x^2$, in order to study
    when a very heavy input can produce a very light minimum relative to
    the chosen scale.
Such a comparison cannot be expressed via condition \textbf{(C)}, which
    corresponds to the intermediate scale $R_\uparrow(x) = x$.

\section{Proofs}\label{sec:proofs}

\begin{proof}[Proof of Theorem~\ref{thm:main_C}]
    First assume that $\xi_1$ satisfies condition \textbf{(C)} (with given
        $\gamma > 0$ and $(s_i, t_i)_{i \geqslant 1}$).
    Let $\xi_2$ be independent of $\xi_1$ with cumulative hazard function
        $R_{\xi_2}(x) = 0$ for $x < s_1$, $R_{\xi_2}(x) = \gamma x$ for $x \in
        [t_i, s_{i + 1})$ and $R_{\xi_2}(x) = \gamma s_i$ for $x \in [s_i,
        t_i)$.
    Then, for $x \in [t_i, s_{i + 1})$, we have
    \begin{align*}
        \frac {R_{\xi_1}(x)}x + \frac {R_{\xi_2}(x)}x \geqslant \frac
            {R_{\xi_2}(x)}x = \gamma,
    \end{align*}
    and, for $x \in [s_i, t_i)$,
    \begin{align*}
        \frac {R_{\xi_1}(x)}x + \frac {R_{\xi_2}(x)}x \geqslant \frac
            {R_{\xi_1}(x)}x \geqslant \gamma,
    \end{align*}
    so $\frac{R_\eta(x)}x \geqslant \gamma$ for all $x \geqslant s_1$
        and therefore $\eta$ is light.
    Further, $\xi_2$ is heavy since
    \begin{align*}
        \frac{R_{\xi_2}(t_i - 0)}{t_i} = \frac{\gamma
            s_i}{t_i} \to 0 \mbox { as } i \to \infty.
    \end{align*}

    {More explicitly, take $x_i = (s_i + t_i)/2 \to \infty$.
    Then $x_i \in [s_i, t_i)$ and
    \[
        \frac{R_{\xi_2}(x_i)}{x_i}
        = \frac{2\gamma s_i}{s_i + t_i}
        \leqslant \frac{2\gamma s_i}{t_i} \to 0.
    \]
    }

    Conversely, assume that $\xi_2$ is a heavy-tailed random variable which is
        independent of $\xi_1$, and that $\eta := \xi_1 \wedge \xi_2$ is light
        tailed.
    Then there exists some $\gamma > 0$ such that
    \begin{align*}
        \gamma \leqslant \liminf_{x \to \infty} \frac {R_{\eta}(x)}x =
            \liminf_{x \to \infty}\big(\frac {R_{\xi_1}(x)}x + \frac
            {R_{\xi_2}(x)}x\big).
    \end{align*}
    Let $u_0 \geqslant 1$ be such that $\big(\frac {R_{\xi_1}(x)}x + \frac
        {R_{\xi_2}(x)}x\big) \geqslant \frac \gamma 2$ for all $x \geqslant
        u_0$.
    {Since $\xi_2$ is heavy-tailed, we may choose $u_i$ recursively
        so that
        \[
            u_i > \ee^{i / 2}u_{i - 1}
            \qquad\text{and}\qquad
            \frac{R_{\xi_2}(u_i)}{u_i} \leqslant \gamma\ee^{-i},
            \qquad i \geqslant 1.
        \]
        This is possible because, for each $i$, the second inequality holds
            at arbitrarily large points.}

    Let $v_i := u_i\ee^{-i / 2}$.
    Then $u_{i - 1} < v_i < u_i$ for every $i \geqslant 1$, so
        the intervals $[v_i, u_i]$ are ordered and disjoint.
    Then, for $v \in [v_i, u_i]$, we have
    \begin{align*}
        \frac{R_{\xi_2}(v)}{v} \leqslant
        \frac{R_{\xi_2}(u_i)}{v_i} \leqslant
        \frac{\gamma \ee^{-i}u_i}{v_i} = \gamma \ee^{-i / 2},
    \end{align*}
    and therefore
    \begin{align*}
        \frac{R_{\xi_1}(v)}{v} \geqslant
        \frac \gamma 2 - \gamma \ee^{-i / 2} \geqslant \frac{\gamma}{10}
    \end{align*}
    for all $i \geqslant 2$.
    {Set $s_j := v_{j + 1}$ and $t_j := u_{j + 1}$ for $j \geqslant
        1$.
        Then $1 < s_1 < t_1 < s_2 < t_2 < \cdots$ tends to infinity and $s_j /
            t_j = \ee^{-(j + 1) / 2} \to 0$.
        The preceding bound therefore gives condition \textbf{(C)} with constant
            $\gamma / 10$.}
\end{proof}

\begin{proof}[Proof of Corollary~\ref{cor:limsup_infty}]
    We show that, under the assumption of the corollary, condition \textbf{(C)}
        holds even for arbitrarily chosen $\gamma > 0$.
    Define, recursively, $t_0 := 1$ and, for $i = 1, 2, \ldots$,
    \begin{align*}
        s_i := \inf\{t \geqslant t_{i - 1} + 1:\, \frac {R_{\xi_1}(t)}
            t \geqslant (i + 1)\gamma\}
        \quad \text{and} \quad
        t_i := (i + 1)s_i.
    \end{align*}
    Clearly, $1 < s_1 < t_1 < s_2 < t_2 < \cdots$ and
        $\lim_{i \to \infty} s_i / t_i = 0$.
    Furthermore, for $x \in [s_i, t_i]$, we have
    \begin{align*}
        \frac {R_{\xi_1}(x)}x \geqslant
        \frac {R_{\xi_1}(s_i)}{t_i}
        \geqslant
        \frac{(i + 1)\gamma s_i}{(i + 1)s_i} = \gamma,
    \end{align*}
    so condition \textbf{(C)} holds.
\end{proof}

{ In Theorem~\ref{thm:main_C}, the linear
    reference scale allows us to control explicitly how long the complementary
    hazard must remain constant for its ratio to the reference scale to become
    small.
    For general hazard scales $R_\downarrow$ and $R_\uparrow$, no such
        predictable rescaling is available: we must instead wait until
        $R_\downarrow$ has grown sufficiently.
    During this time, however, $R_{\xi_1}$ must remain above a fixed positive
        multiple of $R_\uparrow$, even though $R_\uparrow$ may grow irregularly.
    The main difficulty is therefore to identify intervals on which this lower
        bound persists long enough for the complementary hazard $R_{\xi_2}$
        to become small relative to $R_\downarrow$.
    The proof below shows both how to use such intervals and how to recover them
            from the existence of an $R_\downarrow$-heavy complement.
}

\begin{proof}[Proof of Theorem \ref{proposition:main}]
    First, assume condition (ii) (with given $\gamma > 0$
        and $[s_i, t_i]_{i \geqslant 1}$).
    {We let $R_{\xi_2}$ grow only when we need to keep the sum above
        $\gamma R_\uparrow$, and freeze it on the given segments.
        The following running supremum implements this
            rule while preserving monotonicity and right-continuity even when
            the hazards have jumps.
        Define
        \[
            Q(x) := \sup_{y \leqslant x}\bigl(\gamma R_\uparrow(y) -
                R_{\xi_1}(y)\bigr)^+,
            \qquad x \in \mathbb R,
        \]
        where $z^+ := \max\{z, 0\}$.
        The function $Q$ is nondecreasing and satisfies $0 \leqslant
            Q(x) \leqslant \gamma R_\uparrow(x)$, so it is finite and tends
            to zero as $x \to - \infty$.
        It is right-continuous: by right-continuity of the function inside
            the supremum, its values on a sufficiently short interval to
            the right of $x$ do not exceed its value at $x$ by more than
            an arbitrarily small amount.
        Since $\xi_1$ is $R_\downarrow$-heavy and $R_\downarrow \leqslant
            R_\uparrow$ eventually, there are $x_n \to \infty$
            with $R_{\xi_1}(x_n) / R_\uparrow(x_n) \to 0$.
        Thus $Q(x_n) \geqslant (\gamma / 2)R_\uparrow(x_n) \to \infty$ for all
            sufficiently large $n$.
        By monotonicity, $Q(x) \to \infty$ as $x \to \infty$.
        Consequently, $Q$ is a cumulative hazard, and we may take $\xi_2$
            independent of $\xi_1$ with $R_{\xi_2} = Q$.

        For $\eta = \xi_1 \wedge \xi_2$, the choice $y = x$ gives
        \[
            R_\eta(x) = R_{\xi_1}(x) + Q(x) \geqslant \gamma R_\uparrow(x),
        \]
        so $\eta$ is $R_\uparrow$-light.
        On $[s_i, t_i]$ the expression inside the positive part is nonpositive,
            while for $y < s_i$ its positive part is at most $\gamma
            R_\uparrow(s_i - 0)$.
        Hence $\xi_2$ is $R_\downarrow$-heavy, since
        \[
            0 \leqslant \frac{R_{\xi_2}(t_i)}{R_\downarrow(t_i)}
            \leqslant \frac{\gamma R_\uparrow(s_i - 0)}{R_\downarrow(t_i)}
            \underset{i \to \infty}{\longrightarrow}0.
        \]
    }

    Conversely, assume that $\xi_2$ is an $R_\downarrow$-heavy
        random variable, independent of $\xi_1$, and that $\eta = \xi_1 \wedge
        \xi_2$ is $R_\uparrow$-light.
    Then there exists $\gamma > 0$ such that
    \begin{align*}
        \gamma \leqslant \liminf_{x \to \infty}\frac{R_\eta(x)}{R_\uparrow(x)}
        = \liminf_{x \to \infty}
        \left(\frac{R_{\xi_1}(x)}{R_\uparrow(x)} +
            \frac{R_{\xi_2}(x)}{R_\uparrow(x)}\right).
    \end{align*}
    Hence there is $u_0$ with
    \begin{align}
        \label{eq:upper_est}
        \frac{R_{\xi_1}(x)}{R_\uparrow(x)} + \frac{R_{\xi_2}(x)}{R_\uparrow(x)}
        \geqslant \frac{\gamma}{2}
    \end{align}
    for all $x \geqslant u_0$.
    Increasing $u_0$ if necessary, we may also assume that
        $R_\downarrow(x) \leqslant R_\uparrow(x)$ holds for all $x \geqslant
        u_0$ and that $R_\downarrow(u_0) > 1$.
    We now construct the sequence $s_1 < t_1 < s_2 < t_2 < \ldots$ required
        by $(R_\downarrow, R_\uparrow)$-segmentation.

    \emph{Partition the tail $[u_0, + \infty)$ into segments.}
    Fix any strictly decreasing sequence $\{\alpha_n\}_{n = 1}^\infty$
        with $\alpha_n\downarrow 0$ and $\alpha_1 = 1/2$.
    Since $\xi_2$ is $R_\downarrow$-heavy, there exists a sequence $u_0 < u_1 <
        u_2 < \ldots$ diverging to infinity such that
    \begin{align}
        \label{eq:xi_2_estimate}
        R_{\xi_2}(u_i) < \frac{\gamma}{2}\alpha_i\, R_\downarrow(u_i)
        \leqslant \frac{\gamma}{4} R_\uparrow(u_i).
    \end{align}
    This provides the segments $\{[u_i, u_{i + 1}]\}_{i = 0}^\infty$.
    Combining \eqref{eq:xi_2_estimate} with \eqref{eq:upper_est} yields
    \begin{align}
        \label{eq:est_xi1_lower}
        R_{\xi_1}(u_i) > \frac{\gamma}{4} R_\uparrow(u_i).
    \end{align}
    We still need to identify segments that are long enough to satisfy
        \eqref{eq:segmentation_condition}.

    { In the proof
        of Theorem~\ref{thm:main_C}, we start from points $u_i$ where
        $R_{\xi_2}(u_i) / u_i$ is small and choose the left endpoints $v_i = u_i
        e^{-i / 2}$.
        The linear scale lets us control both $v_i / u_i$ and $R_{\xi_2}(u_i) /
            v_i$, giving the required intervals.
        For general hazard scales, choosing endpoints with a small ratio does
            not by itself control $R_\uparrow(s_i - 0) / R_\downarrow(t_i)$
            or ensure that $R_{\xi_1} / R_\uparrow$ stays above a fixed positive
            level between them.
        We therefore look for low values of $R_{\xi_1} / R_\uparrow$ between
            the sampled points and use the last threshold crossing to define
            each segment.
    }

    \emph{Identify candidate segments.}
    We now focus on segments where $R_{\xi_1}$ ``stagnates''.
    Since the collection $\{[u_i, u_{i + 1}]\}_{i = 0}^\infty$ covers $[u_0, +
        \infty)$ and $\xi_1$ is $R_\downarrow$-heavy, there are infinitely many
        indices for which there exists $u_i^* \in (u_i, u_{i + 1})$ with
    \begin{align}
        \label{eq:est_xi1_upper}
        \frac{R_{\xi_1}(u_i^*)}{R_\uparrow(u_i^*)}
        \leqslant \frac{R_{\xi_1}(u_i^*)}{R_\downarrow(u_i^*)} <
            \frac{\gamma}{4}.
    \end{align}
    {Choose an increasing sequence $k_1 < k_2 < \cdots$ of such
        indices, with $k_1 \geqslant 1$.
        By passing to a subsequence if necessary, we may assume that $k_{i +
            1} \geqslant k_i + 2$.
        From now on, we consider only these segments and write
        \[
            u_i^{(l)} := u_{k_i}, \qquad u_i^{(r)} := u_{k_i + 1}, \qquad
                i \geqslant 1.
        \]
    }

    \emph{Constructing $[s_i, t_i]$.}
    {By \eqref{eq:est_xi1_upper}, the ratio is below $\gamma / 4$
        at a point inside each selected segment.
        Define
        \begin{align*}
            s_i := \sup\left\{
            s \in [u_i^{(l)}, u_i^{(r)}] : \frac{R_{\xi_1}(s)}{R_\uparrow(s)} <
                \frac{\gamma}{4}
            \right\},
        \end{align*}
        which we take as the left endpoint.
        Then $u_i^{(l)} < s_i \leqslant u_i^{(r)}$.
        By \eqref{eq:est_xi1_lower}, the ratio is strictly greater than $\gamma
            / 4$ at $u_i^{(r)}$.
        Its right-continuity therefore allows us to choose
        \[
            u_i^{(r)} < t_i < u_{k_i + 2}
        \]
        so close to $u_i^{(r)}$ that the ratio remains greater than $\gamma / 4$
            on $[u_i^{(r)}, t_i]$.
        The definition of $s_i$ and right-continuity at $s_i$ now give
        \begin{align*}
            \frac{R_{\xi_1}(x)}{R_\uparrow(x)} \geqslant \frac{\gamma}{4}
            \quad\text{for all }x \in [s_i, t_i], \ i \geqslant 1.
        \end{align*}
        Moreover, $t_i < u_{k_i + 2} \leqslant u_{k_{i + 1}} < s_{i + 1}$, so
            these intervals are strictly ordered and tend to infinity,
            with $R_\downarrow(s_1) > 1$.
    }

    It remains to verify that the segments are long enough.
    There are points below $s_i$ approaching $s_i$ at which
        the ratio is less than $\gamma / 4$; right-continuity rules out
        an isolated last such point.
        Taking left limits gives
    \[
        \frac{R_{\xi_1}(s_i - 0)}{R_\uparrow(s_i - 0)} \leqslant
        \frac{\gamma}{4}.
    \]
    { Combining the left-limit version
        of \eqref{eq:upper_est} with the preceding bound, and using
        $s_i \leqslant u_i^{(r)} < t_i$, monotonicity of the hazards,
        and \eqref{eq:xi_2_estimate} at $u_i^{(r)} = u_{k_i + 1}$, we obtain
        \[
            \frac{R_\uparrow(s_i - 0)}{R_\downarrow(t_i)}
            \leqslant \frac{4}{\gamma}\, \frac{R_{\xi_2}(s_i -
                0)}{R_\downarrow(t_i)}
            \leqslant \frac{4}{\gamma}\,
                \frac{R_{\xi_2}(u_i^{(r)})}{R_\downarrow(u_i^{(r)})}
            < 2\alpha_{k_i + 1}\underset{i \to \infty}{\longrightarrow}0.
        \]
    }
    This proves \eqref{eq:segmentation_condition} and completes the proof.
\end{proof}

\begin{proof}[Proof of Corollary \ref{Cor2}]
    We construct the sequence of segments $\{[s_i, t_i]\}_{i = 1}^\infty$
        recursively.
    Choose any $t_0$ such that $R_\uparrow(t_0) > 1$, and fix a strictly
        increasing sequence $\{A_i\}_{i = 1}^\infty$ such that $A_1 > 1$
        and $A_i \to \infty$.

    \emph{Choice of $s_i$.}
    Given $t_{i - 1}$, define
    \begin{align*}
        s_i := \inf\Bigl\{
        s \geqslant t_{i - 1} + 1:\ \frac{R_{\xi_1}(s)}{R_\uparrow(s)} \geqslant A_i\gamma
        \Bigr\}.
    \end{align*}
    The existence of $s_i$ follows from \eqref{cond:supremum_condition}.
    Since both $R_{\xi_1}$ and $R_\uparrow$ are right-continuous
        and $R_\uparrow(s_i) > 0$, the ratio $R_{\xi_1} / R_\uparrow$ is
        right-continuous at $s_i$, and therefore
    \begin{align*}
        R_{\xi_1}(s_i) \geqslant A_i\gamma R_\uparrow(s_i).
    \end{align*}

    \emph{Choice of $t_i$.}
    Define $t_i$ as the first time when $R_\uparrow$ reaches the level
        $A_iR_\uparrow(s_i)$:
    \begin{align*}
        t_i := \inf\Bigl\{
        t > s_i:\ R_\uparrow(t) \geqslant A_i R_\uparrow(s_i)
        \Bigr\}.
    \end{align*}
    Such a point is finite and strictly greater than $s_i$
        because $R_\uparrow$ is continuous, nondecreasing and unbounded, and
        $A_i > 1$. Continuity also prevents overshoot at the endpoint:
    {
    \begin{align*}
        R_\uparrow(t_i) = A_iR_\uparrow(s_i), \qquad
        R_\uparrow(x) \leqslant A_iR_\uparrow(s_i)
        \quad\text{for all }x \in [s_i, t_i].
    \end{align*}
    }

    \emph{Verification of $(R_\uparrow, R_\uparrow)$-segmentation.}
    The construction gives $t_{i - 1} + 1 \leqslant s_i < t_i$
        and $R_\uparrow(s_1) > 1$, so the intervals are strictly ordered and
        tend to infinity. They are long enough since
    \begin{align*}
        \frac{R_\uparrow(s_i - 0)}{R_\uparrow(t_i)} =
        \frac{R_\uparrow(s_i)}{R_\uparrow(t_i)}
        \leqslant \frac{1}{A_i}
        \underset{i \to \infty}{\longrightarrow}0.
    \end{align*}
    Second, for every $x \in [s_i, t_i]$,
    \begin{align*}
        \frac{R_{\xi_1}(x)}{R_\uparrow(x)}
        \geqslant
        \frac{R_{\xi_1}(s_i)}{R_\uparrow(x)}
        \geqslant
        \frac{A_i\gamma R_\uparrow(s_i)}{A_iR_\uparrow(s_i)}
        = \gamma,
    \end{align*}
    where we used the monotonicity of $R_{\xi_1}$ and the bound
        on $R_\uparrow(x)$.
    Thus, for the given $\gamma > 0$, we have constructed
        the required segmentation and verified the bound on $R_{\xi_1}$.
\end{proof}

\begin{proof}[Proof of Theorem~\ref{thm:k_out_of_n_segmentation}]
    \emph{(i)$\Rightarrow$(ii).}
    Fix the index set $K := \{1, 2, \dots, k\}$.
    By the assumptions of the theorem, the $k$-fold minimum
    \[
        \xi_1 \wedge \cdots \wedge \xi_k
    \]
    is $R_\uparrow$-light.
    Set
    \[
        \zeta := \xi_2 \wedge \cdots \wedge \xi_k.
    \]
    Then $\zeta$ is independent of $\xi_1$ and is $R_\downarrow$-heavy.
    Moreover,
    \[
        \xi_1 \wedge \zeta = \xi_1 \wedge \cdots \wedge \xi_k
    \]
    is $R_\uparrow$-light.
    Hence, the pair $(\xi_1, \zeta)$ satisfies the hypotheses
        of Theorem~\ref{proposition:main}, and (ii) follows.

    \emph{(ii)$\Rightarrow$(i).}
    Fix $\gamma > 0$ and an $a_0 \geqslant 0$ such that,
        by Theorem~\ref{proposition:main}, there exists an $R_\downarrow$-heavy
        random variable $\xi'$, independent of $\xi_1$, with
    \begin{equation}\label{eq:konn_light_pair_haz}
        R_{\xi_1 \wedge \xi'}(x) = R_1(x) + R'(x) \geqslant \gamma
            R_\uparrow(x),
        \qquad x \geqslant a_0,
    \end{equation}
    where $R_1 := R_{\xi_1}$ and $R' := R_{\xi'}$.

    This is the only point in the argument where condition (ii) is
        used directly.
    The subsequent construction of $\xi_2, \ldots, \xi_n$ is based on $\xi_1$
        and $\xi'$ and follows ideas similar to those in
        \cite{FKT}.
    The hazard $R_1$ is given by the conditions of the theorem.
        We use the auxiliary hazard $R'$ to construct $R_2, \ldots, R_n$ and
            the block boundaries simultaneously, cycling through the sets
            of indices whose minima must be heavy.

    Let $I_1, \ldots, I_M$ denote all $(k - 1)$-subsets of $\{1, \ldots, n\}$,
        where $M = \binom{n}{k - 1}$.
    We arrange these subsets into an infinite periodic sequence
        $\{I_l\}_{l\ge1}$ by repeating them cyclically; that is, $I_i = I_j$
        whenever $i - j$ is a multiple of $M$.

    We construct an increasing sequence of thresholds $a_0 < a_1 < a_2 < \cdots$
        and define the cumulative hazards $R_2, \ldots, R_n$ inductively on each
        interval $[a_{l - 1}, a_l)$.
    Increase $a_0$ if necessary so that $R_\downarrow(a_0) > 0$,
        and let $R_i(x) = 0$ for $x < a_0$ and $i = 2, \ldots, n$.
    On the $l$-th interval, the construction is arranged so that:

    \begin{itemize}
        \item The indices in $I_l$ are called \emph{slow}: their
            hazards are arranged so that their sum, divided by $R_\downarrow$,
            is small at a selected point $b_l$ inside the block.

        \item The indices $i \geqslant 2$ outside $I_l$ are
            \emph{fast} on $[a_{l - 1}, a_l)$, in the sense
            that their hazards dominate $\gamma R_\uparrow$ pointwise.

        \item If $1\notin I_l$, we additionally make the sum
            of the hazards over $I_l$ dominate $R'$.
            Together with $R_1$, this gives the required lower bound for
                the $k$-set $I_l\cup\{1\}$.
    \end{itemize}

    Fix $l \geqslant 1$ and suppose that $R_2, \ldots, R_n$ have
        been defined on $(-\infty, a_{l - 1})$.
        The following formulas specify continuations for $x \geqslant a_{l - 1}$
            using the scheduled set $I_l$.
        We then choose a point $b_l > a_{l - 1}$ satisfying the heaviness
            estimate and set $a_l = b_l + 1$, retaining these continuations
            on $[a_{l - 1}, a_l)$.

    We start by defining the fast indices.

    \emph{Fast indices.}
    For every index $i \in \{2, \ldots, n\}\setminus I_l$ and every $x \in
        [a_{l - 1}, a_l)$, define
    \begin{equation}\label{eq:konn_def_fast}
        R_i(x)
        :=
        \max\{R_i(a_{l - 1} - 0), \gamma R_\uparrow(a_{l -
            1})\}
        + \gamma\bigl(R_\uparrow(x) - R_\uparrow(a_{l - 1})\bigr).
    \end{equation}
    This ensures that, for all such $i$ and all $x \in [a_{l -
        1}, a_l)$,
    \begin{equation}\label{eq:konn_fast_pointwise}
        R_i(x) \geqslant \gamma R_\uparrow(x).
    \end{equation}

    \smallskip
    \emph{Slow indices when $1 \in I_l$.}
    For every $i \in I_l\setminus\{1\}$ and $x \in [a_{l - 1},
        a_l)$, set
    \begin{equation}\label{eq:konn_def_slow_freeze}
        R_i(x) := R_i(a_{l - 1} - 0).
    \end{equation}
    Thus, on this interval, the only hazard within the scheduled slow set that
        may increase is the fixed hazard $R_1$.

    \smallskip
    \emph{Slow indices when $1\notin I_l$ (using the auxiliary
        hazard $R'$).}
    In this case, $I_l\subset\{2, \ldots, n\}$ has cardinality $k - 1$, and
        the critical $k$-tuple is $I_l\cup\{1\}$.
    Within this tuple, the only index outside $I_l$ is $1$,
        whose hazard is fixed.
        We therefore make the sum of the hazards over $I_l$ dominate
            the auxiliary hazard $R'$.

    At the left endpoint $a_{l - 1}$, define nonnegative
        increments from the previously fixed left limits by {
        \begin{equation}\label{eq:konn_catchup}
            \Delta_i := \frac{\bigl(R'(a_{l - 1}) - \sum_{j \in I_l}R_j(a_{l -
                1} - 0)\bigr)_+}{k - 1},
            \qquad i \in I_l,
        \end{equation}
        and set $R_i(a_{l - 1}) := R_i(a_{l - 1} - 0) + \Delta_i$ for $i \in
            I_l$.
    }
    Then, for $x \in [a_{l - 1}, a_l)$ and each $i \in I_l$,
        set
    \begin{equation}\label{eq:konn_def_proxy_split}
        R_i(x)
        :=
        R_i(a_{l - 1})
        + \frac{R'(x) - R'(a_{l - 1})}{k - 1}.
    \end{equation}
    Consequently, for all $x \in [a_{l - 1}, a_l)$,
    \begin{equation}\label{eq:konn_sum_matches_proxy}
        \sum_{i \in I_l}R_i(x)
        =
        \sum_{i \in I_l}R_i(a_{l - 1})
        + \bigl(R'(x) - R'(a_{l - 1})\bigr)
        \geqslant R'(x),
    \end{equation}
    where the inequality follows from \eqref{eq:konn_catchup}.

    \smallskip
    \emph{Choice of $b_l$ and the endpoint $a_l$.}
    We choose $b_l$ sufficiently large so that the scheduled $(k
        - 1)$-fold minimum is small relative to $R_\downarrow$ at $x = b_l$.
    Fix an arbitrary deterministic sequence $\varepsilon_l > 0$
        with $\varepsilon_l \downarrow 0$.
    The argument depends on whether $1 \in I_l$.

    \begin{itemize}
        \item If $1 \in I_l$, then by \eqref{eq:konn_def_slow_freeze},
        \[
            \sum_{i \in I_l}R_i(b_l)
            = R_1(b_l) + \sum_{i \in I_l\setminus\{1\}}R_i(a_{l
                - 1}).
        \]
        Since $\xi_1$ is $R_\downarrow$-heavy, we may choose
            $b_l > a_{l - 1}$ sufficiently large such that
        \begin{equation}\label{eq:konn_choose_endpoint_contains1}
            \frac{R_1(b_l)}{R_\downarrow(b_l)}
                \leqslant \varepsilon_l,
            \qquad
            \sum_{i \in I_l\setminus\{1\}}R_i(a_{l - 1})
            \leqslant \varepsilon_l\, R_\downarrow(b_l).
        \end{equation}

        \item If $1\notin I_l$, then by \eqref{eq:konn_sum_matches_proxy},
        \[
            \sum_{i \in I_l}R_i(b_l)
            \leqslant
            \sum_{i \in I_l}R_i(a_{l - 1}) + R'(b_l).
        \]
        Since $\xi'$ is $R_\downarrow$-heavy, we can again choose
            $b_l > a_{l - 1}$ sufficiently large such that
        \begin{equation}\label{eq:konn_choose_endpoint_no1}
            \frac{R'(b_l)}{R_\downarrow(b_l)}
                \leqslant \varepsilon_l,
            \qquad
            \sum_{i \in I_l}R_i(a_{l - 1})
            \leqslant \varepsilon_l\, R_\downarrow(b_l).
        \end{equation}
    \end{itemize}

    In both cases, such a choice is possible since $R_\downarrow(x) \to \infty$
        and all other quantities involved are already fixed constants.

    Now set $a_l := b_l + 1$.
        The estimate at $b_l$ is preserved by every later step, since $b_l$ is
            strictly inside the completed block.
        Also $a_l > a_{l - 1} + 1$, so the blocks cover $[a_0, \infty)$.

    This completes the inductive construction of the hazards $R_2, \ldots, R_n$
        on $\mathbb{R}$.
    Each function is finite,
        nondecreasing and right-continuous, including at the block boundaries,
        where its value is defined by the new block and is at least its previous
        left limit.
        It is zero below $a_0$.
        Each index $i \geqslant 2$ is fast on infinitely many blocks, so
            \eqref{eq:konn_fast_pointwise} implies $R_i(x) \to \infty$.
        Thus these functions are cumulative hazards.
    Finally, we define independent random variables $\xi_2, \ldots, \xi_n$, also
        independent of $\xi_1$, by
    \[
        \Prob(\xi_i > x) = \exp(-R_i(x)).
    \]

    \emph{$k$-fold minima are $R_\uparrow$-light.}
    Fix a set $K\subset\{1, \ldots, n\}$ with $|K| = k$ and take $x \geqslant
        a_0$.
    Let $l$ be such that $x \in [a_{l - 1}, a_l)$.
    Since $|I_l| = k - 1$, the set $K$ cannot be contained in $I_l$; hence there
        exists an index $j \in K\setminus I_l$.

    If $K\setminus I_l\ne\{1\}$, choose $j \in K\setminus I_l$
        with $j\ne1$.
    Then $j \in \{2, \ldots, n\}\setminus I_l$
        and \eqref{eq:konn_fast_pointwise} yields $R_j(x) \geqslant \gamma
        R_\uparrow(x)$.
    Therefore,
    \[
        R_{ \wedge _{i \in K}\{\xi_i\}}(x)
        = \sum_{i \in K}R_i(x)
        \geqslant \gamma R_\uparrow(x).
    \]

    If $K\setminus I_l = \{1\}$, then necessarily $K =
        I_l\cup\{1\}$, and in particular $1\notin I_l$.
    In this case,
    \[
        R_{ \wedge _{i \in K}\{\xi_i\}}(x)
        = R_1(x) + \sum_{i \in I_l}R_i(x)
        \geqslant R_1(x) + R'(x)
        \geqslant \gamma R_\uparrow(x),
    \]
    where we used \eqref{eq:konn_sum_matches_proxy}
        and
        \eqref{eq:konn_light_pair_haz}.
    The two cases establish the same lower bound for every
        $x \geqslant a_0$.
    Thus, every $k$-fold minimum is $R_\uparrow$-light.

    \emph{$(k - 1)$-fold minima are $R_\downarrow$-heavy.}
    Fix $J\subset\{1, \ldots, n\}$ with $|J| = k - 1$.
    Because the sequence $\{I_l\}$ is periodic, there are infinitely many
        indices $l$ such that $I_l = J$.
    For those $l$, we estimate $\sum_{i \in J}R_i(b_l)$ using
        the way $b_l$ was chosen.

    If $1 \in J$, then \eqref{eq:konn_choose_endpoint_contains1} gives
    \[
        \frac{\sum_{i \in
            J}R_i(b_l)}{R_\downarrow(b_l)}
        \leqslant
        \frac{R_1(b_l)}{R_\downarrow(b_l)}
        + \frac{\sum_{i \in J\setminus\{1\}}R_i(a_{l -
            1})}{R_\downarrow(b_l)}
        \leqslant 2\varepsilon_l \xrightarrow[l \to \infty]{} 0
    \]
    along an infinite subsequence.

    If $1\notin J$, then \eqref{eq:konn_choose_endpoint_no1} yields
    \[
        \frac{\sum_{i \in
            J}R_i(b_l)}{R_\downarrow(b_l)}
        \leqslant
        \frac{\sum_{i \in J}R_i(a_{l - 1})}{R_\downarrow(b_l)}
        + \frac{R'(b_l)}{R_\downarrow(b_l)}
        \leqslant 2\varepsilon_l \xrightarrow[l \to \infty]{} 0
    \]
    along an infinite subsequence.

    In either case,
    \[
        \liminf_{x \to \infty}\frac{R_{ \wedge _{i \in
            J}\{\xi_i\}}(x)}{R_\downarrow(x)} = 0,
    \]
    so $\wedge _{i \in J}\{\xi_i\}$ is $R_\downarrow$-heavy.
    For each $i \geqslant 2$, choose such a set $J$ containing
        $i$.
        Since $0 \leqslant R_i(x) \leqslant \sum_{j \in J}R_j(x)$, each $\xi_i$
            is also $R_\downarrow$-heavy.
    This proves (i) and completes the proof.
\end{proof}

\section*{Statements and Declarations}

\subsection*{Funding}
The study of A.S. Tarasenko was supported by the Program for fundamental
    scientific research of the Siberian Branch of the Russian Academy
    of Sciences, project FWNF-2026-0030.

\subsection*{Competing interests}
The authors declare that they have no competing interests.

\subsection*{Data availability}
Data sharing is not applicable to this article as no datasets were generated
    or analysed during the current study.

\bibliography{sn-bibliography}

\end{document}